\documentclass[12pt]{amsart}
\usepackage{amsthm}
\usepackage{amsmath}
\usepackage{mathrsfs} 
\usepackage{booktabs} 
\usepackage{hyperref}
\usepackage{url}
\usepackage[colorinlistoftodos, shadow]{todonotes}

\title{Connected Quandles with Order Equal to Twice an Odd Prime}
\author{James McCarron}
\address{
Maplesoft\\
615 Kumpf Drive\\
Waterloo, ON\\
CANADA N2V 1K8}
\email{james@maplesoft.com}
\date{\today}

\newtheorem{theorem}{Theorem}
\newtheorem{proposition}[theorem]{Proposition}
\newtheorem{lemma}[theorem]{Lemma}
\newtheorem{example}[theorem]{Example}
\newtheorem{definition}[theorem]{Definition}

\newcommand{\defn}[1]{\textit{\textbf{#1}}}

\newcommand{\binomial}[2]{\ensuremath{\genfrac{(}{)}{0pt}{}{#1}{#2}}}

\newcommand{\intersect}{\ensuremath{\cap}}

\newcommand{\integers}{\ensuremath{\mathbb{Z}}}

\newcommand{\alt}[1]{\ensuremath{A_{#1}}}
\newcommand{\symm}[1]{\ensuremath{S_{#1}}}
\newcommand{\intmod}[1]{\ensuremath{\integers_{#1}}}

\DeclareMathOperator{\Aut}{Aut}
\DeclareMathOperator{\Inn}{Inn}

\DeclareMathOperator{\Conj}{Conj}

\newcommand{\PSL}[2]{\ensuremath{\mathrm{PSL}(#1,#2)}}

\newcommand{\PGL}[2]{\ensuremath{\mathrm{PGL}(#1,#2)}}
\newcommand{\AGL}[2]{\ensuremath{\mathrm{AGL}(#1,#2)}}
\newcommand{\PGAMMAL}[2]{\ensuremath{\mathrm{P}\Gamma\mathrm{L}(#1,#2)}}
\newcommand{\PSp}[2]{\ensuremath{\mathrm{PSp}(#1,#2)}}
\newcommand{\Sp}[2]{\ensuremath{\mathrm{Sp}(#1,#2)}}
\newcommand{\POmega}[2]{\ensuremath{\mathrm{P\Omega}(#1,#2)}}
\newcommand{\POmegaPlus}[2]{\ensuremath{\mathrm{P\Omega^{+}}(#1,#2)}}
\newcommand{\OmegaPlus}[2]{\ensuremath{\mathrm{\Omega^{+}}(#1,#2)}}
\newcommand{\PSU}[2]{\ensuremath{\mathrm{PSU}(#1,#2)}}

\newcommand{\F}[1]{\ensuremath{\mathbb{F}_{#1}}}

\newcommand{\centre}[1]{\ensuremath{Z(#1)}}
\newcommand{\socle}[1]{\ensuremath{\mathrm{Socle}(#1)}}

\newcommand{\qua}{\ensuremath{\triangleright}}

\newcommand{\q}{\qua}

\let\iso\simeq

\begin{document}

\begin{abstract}
We show that there is an unique connected quandle of order twice an odd prime
number greater than $3$.
It has order $10$ and is isomorphic to the conjugacy class of transpositions
in the symmetric group of degree $5$.
This establishes a conjecture of L.~Vendramin.
\end{abstract}

\maketitle

\section{Introduction}\label{sect:intro}

Quandles were introduced independently in 1982 by D. Joyce~\cite{Joyce1982a}
and S. Matveev~\cite{Matveev1982} as invariants of knots.
To each knot one associates a (generally) non-associative algebraic
system called the knot quandle which turns out to be a very strong
algebraic invariant.

The number $q_{n}$ of isomorphism classes of quandles of order $n$ is
known to grow very quickly with $n$.
This was already evident with the computational determination of the number
of quandles of small orders \cite{HoNelson2005,HendersonMacedoNelson2006,OEIS:A181769}.
Recently, Blackburn showed that $q_{n}$ grows like $2^{n^{2}}$,
asymptotically~\cite{Blackburn2012}.
Because the complete set of quandles of even small orders appears to be
intractably large there has, in recent years, been considerable interest
in counting and constructing quandles of more restricted classes.
Connected quandles are of particular importance because knot quandles are
connected, and homomorphic images of connected quandles are connected.
Therefore, the finite quandles that appear as homomorphic images of knot
quandles are necessarily connected, and it is these quandles that figure in
computable invariants of knots
\cite{FennRourke1992,Kamada2002,Carter2010}.
Thus, the connected quandles of prime order,
and of order equal to the square of a prime,
have been determined~\cite{EtingofGuralnickSoloviev2001,Grana2004}.

Clauwens~\cite{Clauwens2011} computed the connected quandles up to order $14$ and showed,
in particular, that no connected quandles of order $14$ exist.
L. Vendramin computed the connected quandles to order $35$~\cite{Vendramin2011}.
This is sequence A181771~\cite{OEIS:A181771} in the
Online Encyclopaedia of Integer Sequences~\cite{OEIS}.

From this data, it may be observed that, apart from $n = 2$, the values for which
\emph{no} connected quandles appear are the numbers
$14 = 2\cdot 7$,
$22 = 2\cdot 11$,
$26 = 2\cdot 13$
and
$34 = 2\cdot 17$.
Each is equal to twice an odd prime number.
Moreover, these are all the numbers of the form $2p$,
with $p > 5$ and $2p \leq 35$.
There is, however, a connected quandle of order $10 = 2\cdot 5$.

\begin{example}[Connected Quandle of Order $10$]
The conjugacy class of transpositions in the symmetric
group $S_{5}$ of degree $5$ has length $10 = 2\cdot 5$.
Regarded as a quandle under the operation of conjugation,
it is simple and therefore connected.
\end{example}

These observations suggest our main theorem,
which establishes a conjecture of L.~Vendramin~\cite{Vendramin2012}.
\begin{theorem}\label{thm:main}
Let $Q$ be a connected quandle of order $2p$,
where $p > 3$ is a prime number.
Then $Q$ has order $10$ and is isomorphic to the
quandle of the conjugacy class of transpositions in the
symmetric group of degree $5$.
\end{theorem}

Our strategy for the proof is as follows.
First, we use an important theorem of Clauwens~\cite{Clauwens2011}
to show that our quandle is simple.
The importance of this is that we have quite a lot
of information about the structure of the inner
automorphism group, thanks to \cite{AndruskiewitschGrana2003}.
From this description of the inner automorphism group,
we construct a faithful permutation representation of
degree $2p$ on a conjugacy class.
Now, this action of the inner automorphism group may,
or may not, be primitive, and we analyse these two
possibilities separately.
For the primitive case, we need to know the primitive
groups of degree $2p$.
We derive a list of primitive groups of degree $2p$
from a result of Liebeck and Saxl.~\cite{LiebeckSaxl1985a}
Having determined these, the conclusion follows quite
easily.
For the imprimitive case, we have to work a bit harder.
We construct a different faithful permutation representation
of the inner automorphism group of prime degree $p$.
Using this, we are able to conclude that the inner
automorphism group is, in fact, doubly transitive
with simple socle, and to construct a subgroup of
index $p$ in the socle.
We then use an observation due to D.~Holt on point stabilisers
in doubly transitive groups to conclude that this case cannot occur.
The conclusion, then, is that there are no
imprimitive examples so that the result of the
primitive case applies, and we arrive at
the apocalyptic conclusion of the theorem.

The remainder of the paper is organised as follows.
We gather some relevant background material in Section~\ref{sect:prelim}.
Then, in Section~\ref{sect:prim}, we prove our main theorem for primitive quandles.
Section~\ref{sect:imprim2p} deals with the case of an imprimitive quandle,
and Section~\ref{sect:pfprim2p} contains the proof of Proposition~\ref{prop:prim2p}
which classifies the primitive groups of degree equal to twice an odd
prime.

\subsection*{Acknowledgement}
The author thanks Erik Postma and Leandro Vendramin for reading earlier
drafts of this paper.

\section{Preliminaries}\label{sect:prelim}

Let us begin by defining the principal objects of study.
\begin{definition}[Quandle]\label{def:quandle}
A \defn{quandle} is a set $Q$ together with a binary operation
$\q : Q\times Q\to Q$ which satisfies the following axioms.
\begin{enumerate}
\item[(Q1)]{For all $a$ and $b$ in $Q$, there is a unique $x$ in $Q$ such that $b = x\q a$.}
\item[(Q2)]{For all $a$, $b$ and $c$ in $Q$, we have $(a\q b)\q c = (a\q c)\q(b\q c)$.}
\item[(Q3)]{For all $a$ in $Q$, we have $a\q a = a$.}
\end{enumerate}
\end{definition}

We give several standard examples of quandles.

\begin{example}[Conjugation Quandle]\label{ex:conjq}
Let $G$ be a group and, for $a$ and $b$ in $G$,
define $a\q b = b^{-1}ab$.
Then the system $\langle G,\q\rangle$ is a quandle, called the
\defn{conjugation quandle} $\Conj{G}$ of $G$.
Moreover, any conjugacy class, or union of conjugacy classes of $G$ forms a quandle
with conjugation as the quandle operation.
\end{example}

The \defn{order} of a quandle is the cardinality of its underlying set.

\begin{example}[Trivial Quandle]\label{ex:trivq}
The \defn{trivial quandle} on a set $Q$ is defined by the
binary operation $\q$ for which $a\q b = b$, for all $a$ and $b$ in $Q$.
This is the only associative quandle operation on a set $Q$.
Any two trivial quandles of the same order are isomorphic,
and we denote the trivial quandle of order $n$ by $T_{n}$.
\end{example}

\begin{example}[Affine Quandle]\label{ex:alexq}
Let $A$ be an Abelian group, and let $\tau$ be an automorphism of $A$.
We obtain a quandle structure on the underlying set of $A$ by defining,
for $a$ and $b$ in $A$,
$a\q b = \tau a + (1 - \tau)b$.
A quandle of this form is called an \defn{Alexander quandle}
(or an \defn{affine quandle}).
\end{example}


\begin{example}[Dihedral Quandle]\label{ex:dihedralq}
The \defn{dihedral quandle} of order $n$, where $n$ is a positive integer,
is defined to be the set $\intmod{n}$ of integers modulo $n$,
together with the binary operation $\q$ defined by
$a\q b = 2b - a\pmod{n}$, for all $a$ and $b$ in $\intmod{n}$.
\end{example}


Homomorphisms, automorphisms and subquandles are defined in the
natural way.
Thus, if $Q$ and $R$ are quandles, then a map $\varphi : Q\to R$ is
a quandle \defn{homomorphism} if $\varphi(a\q b) = (\varphi a)\q(\varphi b)$
for all $a$ and $b$ in $Q$.
A homomorphism is an \defn{isomorphism} if it is bijective,
and an \defn{automorphism} of a quandle $Q$ is an isomorphism
$Q\to Q$.
The set of all automorphisms of a quandle $Q$ forms a group $\Aut Q$.

Let $Q$ be a quandle.
Because of the quandle axiom (Q1), the right translation mappings
\begin{displaymath}
\rho_{a} : Q\to Q : q\mapsto q\q a,
\end{displaymath}
for $a$ in $Q$, are bijective.
Furthermore, axiom (Q2) guarantees that the map $\rho_{a}$ is
an endomorphism of $Q$.
For, if $x$ and $y$ belong to $Q$, then we have
\begin{displaymath}
\rho_{a}(x\q y) = (x\q y)\q a = (x\q a)\q(y\q a) = (\rho_{a}x)\q(\rho_{a}y).
\end{displaymath}
Therefore, each right translation map $\rho_{a}$ is an automorphism of $Q$.

The set $\{\rho_{a} : a\in Q\}$ of right translations does not typically
form a subgroup of the symmetric group on $Q$,
but the subgroup generated by these maps is of great importance in the theory
of quandles.
\begin{definition}[Inner Automorphism Group]
Let $Q$ be a quandle.
The group $\langle\rho_{q} : q\in Q\rangle$
generated by the right translation maps $\rho_{q}$,
for $q$ in $Q$, is called the \defn{inner automorphism group} of $Q$,
and is denoted by $\Inn Q$.
\end{definition}
We frequently think of the inner automorphism group of a quandle as a
subgroup of the symmetric group on the underlying set of the quandle.
In this way, we can apply directly the theory of permutation groups to
the group of inner automorphisms.
If, as is often the case, a quandle is represented by its Cayley table,
then the right translation maps which generate the inner automorphism group
can be read off of the Cayley table directly, as they form its columns.

Of considerable interest are the ``connected'' quandles, which we define presently.
\begin{definition}[Connected Quandle]\label{def:connq}
A quandle is \defn{connected} if its inner automorphism group acts
transitively on the quandle.
\end{definition}

This paper is really about simple quandles, formally defined as follows.
\begin{definition}[Simple Quandle]
A quandle is \defn{simple} if it has more than one element and
its only proper homomorphic image is the singleton quandle.
\end{definition}

It is easy to see that a simple quandle is connected,
but there are connected quandles that are not simple.

\begin{example}[A Non-Simple Connected Quandle]
There are two connected quandles of order equal to $6$.
They are given by the Cayley tables Table~\ref{tab:conn6a} and Table~\ref{tab:conn6b}
\begin{table}
\caption{Cayley table of the first connected quandle of order $6$}\label{tab:conn6a}
\begin{tabular}{c|cccccc}
\q&$a$&$b$&$c$&$d$&$e$&$f$\\\hline
$a$&$a$&$a$&$d$&$c$&$f$&$e$\\
$b$&$b$&$b$&$e$&$f$&$c$&$d$\\
$c$&$d$&$e$&$c$&$a$&$b$&$c$\\
$d$&$c$&$f$&$a$&$d$&$d$&$b$\\
$e$&$f$&$c$&$b$&$e$&$e$&$a$\\
$f$&$e$&$d$&$f$&$b$&$a$&$f$
\end{tabular}
\end{table}
\begin{table}
\caption{Cayley table of the second connected quandle of order $6$}\label{tab:conn6b}
\begin{tabular}{c|cccccc}
\q&$a$&$b$&$c$&$d$&$e$&$f$\\\hline
$a$&$a$&$a$&$d$&$e$&$f$&$c$\\
$b$&$b$&$b$&$f$&$c$&$d$&$e$\\
$c$&$f$&$d$&$c$&$a$&$c$&$b$\\
$d$&$c$&$e$&$b$&$d$&$a$&$d$\\
$e$&$d$&$f$&$e$&$b$&$e$&$a$\\
$f$&$e$&$c$&$a$&$f$&$b$&$f$
\end{tabular}
\end{table}
(where we take the underlying set, in each case, to be
$\{ a, b, c, d, e, f \}$).
Neither quandle is simple, however,
as each admits a homomorphism onto the
(unique) connected quandle with three elements.
\end{example}

Let us now turn our attention to specific background needed for the
proof of our main result.
We begin by noting that Theorem~\ref{thm:main} has been proved,
computationally, for primes $p < 19$.
\begin{proposition}[\cite{Vendramin2011}]\label{prop:truesmall}
Theorem~\ref{thm:main} is true for $p\leq 17$.
\end{proposition}
\begin{proof}
This is a statement of the computational results from \cite{Vendramin2011},
from which connected quandles are known up to order $35$.
We note only that the present author has independently replicated
Vendramin's results up to order $30$ (see \cite{McCarron2012a}).
\end{proof}

We quote the following result of Clauwens, which is the starting point
for our investigations.
\begin{theorem}[\cite{Clauwens2011}]\label{thm:simp2p}
If $p$ is a prime and $p > 3$, then a connected quandle
of order $2p$ is simple.
\end{theorem}


In \cite{AndruskiewitschGrana2003}, Andruskiewitsch and Gra\~{n}a
described the structure of the inner automorphism group of a
simple quandle.
We summarise the results from ~\cite{AndruskiewitschGrana2003}
that we need in the following theorem.

\begin{theorem}[\cite{AndruskiewitschGrana2003}]\label{thm:simpledesc}
Let $Q$ be a simple quandle, and let $G = \Inn Q$ be its inner
automorphism group.
Suppose that the order of $Q$ is not a prime power.
Then:
\begin{enumerate}
\item[(a)]{every proper quotient of $G$ is cyclic;}
\item[(b)]{the centre $\centre{G}$ of $G$ is trivial;}
\item[(c)]{the map $\rho : Q\to G : q\mapsto\rho_{q}$ is injective, where $a\rho_{q} = a\q q$, for all $a\in Q$;}
\item[(d)]{$C = Q\rho$ is a single conjugacy class in $G$, and $G = \langle C\rangle$ (that is, $C$ generates $G$,
and we can identify $Q$ with the conjugacy class $C$ in $G$); and,}
\item[(e)]{$G$ has a unique minimal normal subgroup
\begin{displaymath}
D = [G,G] = T_{1}\times T_{2}\times\cdots\times T_{k},
\end{displaymath}
for some $k\geq 1$, where each subgroup $T_{i}$ is isomorphic to a
finite non-abelian simple group $T$.}
\end{enumerate}
\end{theorem}
We note that \cite{AndruskiewitschGrana2003} also describes the
structure of $\Inn Q$ for a simple quandle $Q$ of prime power order,
but we do not need those results here.

We observe that, under the identification of the quandle $Q$
with the conjugacy class $C$ in $G$, the actions of $G$ on
$Q$ by automorphisms and on $C$ by conjugation, are equivalent.
For, given arbitrary elements $a$ and $b$ in $C$, and any element
$g\in G$, the conjugates $a^{g}$ and $b^{g}$ belong to $C$,
and we have
\begin{displaymath}
(a^{g}\q b^{g})
= (b^{g})^{-1}a^{g}b^{g}
= (g^{-1}bg)^{-1}g^{-1}agg^{-1}bg
= g^{-1}b^{-1}abg
= (a\q b)^{g}.
\end{displaymath}

We shall use the following result from \cite{McCarron2012a}.
\begin{lemma}[\cite{McCarron2012a}]\label{noconnthreetrans}
A finite quandle with at least four members
and with a triply transitive group of automorphisms
is trivial.
\end{lemma}
Note that a quandle can have a doubly transitive
automorphism group~\cite{FermanNowikTeicher2008}.

\section{Primitive Quandles}\label{sect:prim}

Let $Q$ be a quandle with inner automorphism group $G = \Inn Q$.
If $Q$ is connected, then $G$ acts (by definition) transitively
on $Q$.
However, the action of $G$ on $Q$ may, or may not, be primitive.

\begin{definition}[Primitive and Imprimitive Quandles]
A connected quandle is said to be \defn{primitive} if its inner automorphism
group acts primitively on it.
A connected quandle is \defn{imprimitive} if its inner automorphism
group acts imprimitively on it.
\end{definition}

By considering the contrapositive, it is easy to see that a
primitive quandle is simple.
However, there do exist simple, imprimitive quandles.
\begin{example}[A Simple, Imprimitive Quandle]
The conjugacy class of the $5$-cycle $(1,2,3,4,5)$ in the alternating group
$\alt{5}$ of degree $5$ is a simple quandle of order $12$,
but its inner automorphism group, which is $\alt{5}$,
does not act primitively on it.
\end{example}

In the remainder of this section, we shall prove Theorem~\ref{thm:main}
for primitive quandles.
To this end, we shall need the following classification of primitive groups
of degree $2p$, for an odd prime $p$.

\begin{proposition}\label{prop:prim2p}
Let $G$ be a finite primitive permutation group of degree
$2p$, where $p$ is an odd prime, and suppose that
$\alt{2p}\not\leq G$.
Let $S = \socle{G}$.
Then $G$ is either soluble of degree $p$, and $G\leq\AGL{1}{p}$,
or $G$ is an almost simple group among the following cases:
\begin{enumerate}
\item[(1)]{$S = \alt{5}$ acting on $2$-sets, of degree $10$ ($p = 5$);}\label{case:alt5}
\item[(2)]{$S = M_{22}$ of degree $22$ ($p = 11$).}\label{case:m22}
\item[(3)]{$S = \PSL{2}{q}$ in its natural action of degree $q + 1$ on the projective line,
where $q$ is an odd prime, and $p = \frac{q+1}{2}$ is prime;}\label{case:psl2q}
\item[(4)]{$S = \PSL{2}{5}$ acting on cosets of a dihedral subgroup of degree $10$ ($p = 5$);}\label{case:psl25}
\item[(5)]{$S = \PSL{2}{4}$ acting on cosets of a dihedral subgroup of degree $6$ or $10$ ($p\in\{ 3, 5 \}$);}\label{case:psl24a}
\item[(6)]{$S = \PSL{2}{4}$ acting on cosets of $\PGL{2}{2}$, of degree $10$ ($p = 5$);}\label{case:psl24b}
\item[(7)]{$S = \Sp{4}{2}$, of degree $6$ or $10$ (two actions) ($p\in\{ 3, 5 \}$).}\label{case:sp42}
\end{enumerate}
\end{proposition}

The proof of Proposition~\ref{prop:prim2p} will be given below in
Section~\ref{sect:pfprim2p}.

We now proceed to prove our main result for primitive quandles.

\begin{theorem}\label{thm:primq2p}
Let $Q$ be a primitive quandle of order $2p$, where $p$ is an odd prime.
Then $Q$ is isomorphic to the quandle of transpositions in the symmetric
group of degree $5$.
\end{theorem}
\begin{proof}
We may (and do) suppose that $p > 17$, by Proposition~\ref{prop:truesmall}.

Since $Q$ is primitive it is, by definition, connected.
By Theorem~\ref{thm:simp2p}, $Q$ is simple.
Let $G = \Inn Q$ be the inner automorphism group of $Q$.
Since the order of $Q$ is a not a prime power,
we have from
Theorem~\ref{thm:simpledesc}
that $G$
is a non-abelian group whose proper quotients are cyclic,
and $G$ has an unique minimal normal subgroup $D$ isomorphic
to a direct power of a non-abelian finite simple group $T$.
Furthermore, $G$ has a generating conjugacy class $C$,
of length $2p$, such that $Q$ is isomorphic to the conjugation
quandle defined on the conjugacy class $C$.
Finally, the action of $G$ on $Q$ is permutation isomorphic
to the action of $G$ on $C$ by conjugation.

By hypothesis, $Q$ is a primitive quandle,
so the action of $G$ on $C$ is primitive.
If $G$ has alternating socle (in its natural action) then,
since $2p > 5$, it follows that $G$ is (at least) triply transitive.
This case is excluded by Lemma~\ref{noconnthreetrans}.
From the supposition that $p > 17$,
and the classification of primitive groups of degree $2p$,
we see that $G$ is an almost simple group with socle
$\PSL{2}{q}$, for $q$ a power of an odd prime,
acting naturally on $1$-dimensional subspaces of
$\F{q}^{2}$.
Thus, $\PSL{2}{q}\leq G\leq \PGAMMAL{2}{q}$ and so,
if $H$ is the stabiliser of a point, then $H$ has trivial centre,
by \cite[Lemma 7]{EtingofGuralnickSoloviev2001}.
But this means that $H$ cannot be the centraliser of any element $x\in C$,
since every such element $x$ belongs to its own centraliser.
This completes the proof.
\end{proof}

\section{Imprimitive Quandles}\label{sect:imprim2p}

We consider in this section the case of an imprimitive quandle $Q$,
by showing that none of order equal to twice an odd prime exist.

We shall need a number of results on finite permutation groups.
The following result is due to Burnside.
\begin{theorem}[\cite{Burnside1901}]\label{thm:burnside}
A transitive permutation group of prime degree is either
soluble or doubly transitive.
\end{theorem}


We also need the following result.
The author thanks Derek Holt for explaining his proof
of this result.
\begin{lemma}[\cite{MO75672}]\label{lem:stabsub2cent}
Let $G$ be a doubly transitive group of prime degree $p$,
acting on a set $\Omega$.
Let $H$ be the stabiliser of a point in $\Omega$,
and let $N$ be a subgroup of index $2$ in $H$.
Then the centre of $N$ is trivial.
\end{lemma}


\begin{proposition}
A connected quandle of order $2p$, where $p$ is an odd prime,
is primitive.
\end{proposition}
\begin{proof}
Let $Q$ be a connected quandle of order $2p$,
where $p > 17$ is an odd prime.
Again, by Clauwen's Theorem~\ref{thm:simp2p}, $Q$ is simple.
As in the primitive case, $G = \Inn Q$ has a generating conjugacy class $C = x^{G}$,
where $x \in G$,
of length $\mid C\mid = 2p$,
and an unique minimal normal subgroup
$D = [G,G] = T_{1}\times T_{2}\times\cdots\times T_{k}$,
with each $T_{i}\iso T$, a finite non-abelian simple group.
(Hence $D = \socle{G}$.)

Suppose, for an eventual contradiction, that $Q$ is imprimitive;
that is, (after identifying $Q$ with $C$) the action of $G$ on $C$ is imprimitive.

Since $G$ acts imprimitively on $C$, the centraliser
$Z := Z_{G}(x)$ is not maximal in $G$.
Therefore, there is a subgroup $M$ of $G$ such that
\begin{displaymath}
Z < M < G,
\end{displaymath}
with each inclusion proper.
Since the index $[G:Z] = 2p$, it follows that either $[G:M] = 2$
or $[G:M] = p$.
If $M$ has index $2$ in $G$, then $M$ is normal in $G$.
But then, since $x\in Z\leq M$, it follows that $M$ contains
$C = x^{G}$.
Since $C$ generates $G$, we have $M = G$, a contradiction.
Therefore,
\begin{displaymath}
[G:M] = p.
\end{displaymath}
Since $[G:Z] = 2p$, it follows that $[M:Z] = 2$
(and so, $Z$ is normal in $M$).

Now, $M\leq MD\leq G$, and $MD$ is a subgroup of $G$ since $D$ is normal,
so either $MD = M$ or $MD = G$, by the maximality of $M$ in $G$.


Suppose, first, that $MD = M$, so that $D\leq M$.

Now, $Z\leq ZD\leq M$, so either $ZD = Z$ or $ZD = M$,
because $Z$ is maximal in $M$.

Suppose that $ZD = Z$; then $D\leq Z$, so that $D$ commutes with $x$.
Let $y\in C$ and choose $g\in G$ such that $y = x^{g}$.
Then
\begin{displaymath}
D = D^{g}\leq Z^{g} = Z_{G}(x^{g}) = Z_{G}(y).
\end{displaymath}
Since $y\in C$ was arbitrary, it follows that
\begin{displaymath}
D\leq\bigcap_{c\in C} Z_{G}( c ) \leq Z(G) = 1,
\end{displaymath}
since $G = \langle C\rangle$.
This is a contradiction, so $ZD\neq Z$, and therefore $ZD = M$.

Since $ZD = M$, we have
\begin{displaymath}
\mid M\mid = \mid DZ\mid = \frac{\mid D\mid\mid Z\mid}{\mid D\intersect Z\mid}.
\end{displaymath}
Using $\mid M\mid = \frac{\mid G\mid}{[G:M]} = \frac{\mid G\mid}{p}$,
and $\mid Z\mid = \frac{\mid G\mid}{[G:Z]} = \frac{\mid G\mid}{2p}$,
we obtain
\begin{displaymath}
\frac{\mid G\mid}{p} = \frac{\mid D\mid\mid G\mid}{2p\mid D\intersect Z\mid};
\end{displaymath}
whence
\begin{displaymath}
2 = \frac{\mid D\mid}{\mid D\intersect Z\mid} = [D : D\intersect Z].
\end{displaymath}
Consequently, $D\intersect Z$ is a normal subgroup of index $2$ in $D$.
But $D$ is a direct power of a non-abelian simple group, so this is impossible.
(The normal subgroups of $D = T_{1}\times T_{2}\times\cdots\times T_{k}$
are all of the form $\prod_{i\in I}T_{i}$, for some subset $I$ of $\{1,2,\ldots, k\}$.)


Consequently, we must have $G = MD$.

From the formula
\begin{displaymath}
\mid G\mid = \frac{\mid M\mid\mid D\mid}{\mid D\intersect M\mid}
\end{displaymath}
we get
\begin{displaymath}
[D : D\intersect M] = p.
\end{displaymath}
In particular, $D\intersect M$ is properly contained in $D$.

The core $M_{G}$ of $M$ in $G$ is a normal subgroup of $G$,
so the intersection $D\intersect M_{G}$ is a normal subgroup
of $G$ contained in $D$.
By the minimality of $D$, we must therefore have either
$D\intersect M_{G} = 1$ or $D\intersect M_{G} = D$.
But $D\intersect M_{G}\leq D\intersect M < D$,
so $D\intersect M_{G} = 1$.

Now, since $M_{G}$ is normal in $G$, if $M_{G}\neq 1$,
then $M_{G}$ contains a minimal normal subgroup of $G$
disjoint from $D$.
But this contradicts the uniqueness of $D$.
Therefore, $M_{G} = 1$, and $G$ acts faithfully on the
cosets of $M$ in $G$.
This means that $G$ is a transitive group of degree
$p = [G:M]$.
By Burnside's Theorem~\ref{thm:burnside} (since $G$ is insoluble),
this action of $G$ on the cosets of $M$ is doubly transitive.
Since $G$ is insoluble, it is almost simple and,
in particular, $D$ is a simple group.
Now Lemma~\ref{lem:stabsub2cent} supplies a final
contradiction, since $Z$, being a centraliser of $x$,
has nontrivial centre, and has index equal to $2$ in
the point stabiliser $M$.



This completes the proof.
\end{proof}

\section{Proof of Proposition~\protect\ref{prop:prim2p}}\label{sect:pfprim2p}

Our proof of Proposition~\ref{prop:prim2p} is based on the following result
of M. W. Liebeck and J. Saxl.

\begin{theorem}[\cite{LiebeckSaxl1985a}]
Let $G$ be a primitive permutation group of degree $mp$,
where $p$ is a prime and $m < p$, and assume that $G$ does
not contain $A_{mp}$.
Then either $G$ is soluble or $G$ is one of the groups in
\cite[Table 3]{LiebeckSaxl1985a}.
\end{theorem}

We do not reproduce Table~3 from \cite{LiebeckSaxl1985a},
though we do use it to analyse the various cases that arise.
Instead, we simply give a description for the corresponding case.
The groups that occur are almost simple groups, and they are
described according to the action of their socles in \cite[Table 3]{LiebeckSaxl1985a}.
We have also organised the various cases into sections, as follows.

\subsection{Alternating Socle}
There are, of course, for each odd prime $p$, the primitive groups
with socle the alternating group $\alt{2p}$ in its natural action.

\paragraph{Case 1}
The only other alternating groups that occur have degrees
$15$, $35$ or $\frac{c(c-1)}{2}$, where $p\in\{ c, c - 1 \}$.
The only primes for which the Diophantine equation $\binomial{c}{2} = 2p$
has a solution are $p = 3$ and $p = 5$.
This gives rise to Case (1) in Proposition~\ref{prop:prim2p}.

\subsection{Classical Socle}
There are a variety of cases involving classical groups.

\subsubsection{PSL}
There are a number of cases in which the socle is a projective
special linear group.

\paragraph{Case 2}
There is an action of $\PSL{d}{q}$ on $1-$ or $(d-1)$-dimensional subspaces.
The degrees are of the form
\begin{displaymath}
n = \frac{(q^{d} - 1)}{(q - 1)},
\end{displaymath}
where $p$ divides $n$ and $d\geq 2$.

Using $2p = n$, we have
\begin{displaymath}
2p = \frac{(q^{d} - 1)}{(q - 1)} = 1 + q + \cdots + q^{d-1}.
\end{displaymath}
If $q$ is even, then $1 + q + \cdots + q^{d-1}$ is odd, so $q$ must be odd.
Now, $1 + q + \cdots + q^{d-1}$ is an even sum of odd terms, so the number
$d$ of summands must be even.
Write $d = 2\delta$.

Suppose that $d > 2$, so that $\delta > 1$.
Then
\begin{displaymath}
q^{d} - 1 = (q^{\delta})^{2} - 1 = (q^{\delta} - 1)(q^{\delta} + 1).
\end{displaymath}
Hence,
\begin{displaymath}
2p = \frac{(q^{\delta} + 1)(q^{\delta} - 1)}{(q - 1)} = (q^{\delta} + 1)(1 + q + \cdots + q^{\delta - 1}).
\end{displaymath}
Since $q\geq 2$, we have
\begin{displaymath}
1 + q + \cdots + q^{\delta - 1} \geq 1 + q \geq 3.
\end{displaymath}
Hence, $1 + q + \cdots + q^{\delta - 1} = p$ and $q^{\delta} + 1 = 2$.
But then, $q^{\delta} = 1$ and $\delta = 0$, a contradiction.
Therefore, $d = 2$, and we get
\begin{displaymath}
2p = \frac{(q^{2} - 1)}{(q - 1)} = q + 1,
\end{displaymath}
and so $\frac{(q + 1)}{2}$ is a prime.
This yields Case (3) in Proposition~\ref{prop:prim2p}.

\paragraph{Case 3}

Next, there is an action of $\PSL{d}{q}$ on $2$- or $(d-2)$-dimensional subspaces,
where the degrees are of the form
\begin{displaymath}
n = \frac{(q^{d} - 1)(q^{d-1} - 1)}{(q^{2} - 1)(q - 1)},
\end{displaymath}
and where $d\geq 4$, and either $p = \frac{(q^{d-1} - 1)}{(q-1)}$
or $p$ divides $\frac{(q^{d} - 1)}{(q-1)}$.

To handle this case, suppose first that $p = \frac{(q^{d-1} - 1)}{(q-1)}$.
Then, using $2p = n$, we obtain
\begin{displaymath}
2\frac{(q^{d-1} - 1)}{(q-1)} = \frac{(q^{d} - 1)(q^{d-1} - 1)}{(q^{2} - 1)(q - 1)},
\end{displaymath}
which yields
\begin{displaymath}
2 = \frac{q^{d} - 1}{q^{2} - 1},
\end{displaymath}
or
\begin{displaymath}
2(q^{2} - 1) = q^{d} - 1.
\end{displaymath}
Dividing by $q - 1$, we obtain
\begin{displaymath}
2q + 2 = 2(q + 1) = 1 + q + \cdots + q^{d-1}.
\end{displaymath}
so that
\begin{displaymath}
q + 1 = q^{2} + \cdots + q^{d-1}.
\end{displaymath}
But, from $q\geq 2$ it follows that $q^{2} > q$, and
\begin{displaymath}
q + 1 < q^{2} + 1 < q^{2} + q^{3} + \cdots + q^{d-1},
\end{displaymath}
unless $d - 1 = 2$, so that $d = 3$.
But we assumed that $d\geq 4$, so this cannot be the case.

Now suppose that $p$ divides $\frac{(q^{d} - 1)}{(q - 1)}$,
and write
\begin{displaymath}
sp = \frac{(q^{d} - 1)}{(q - 1)},
\end{displaymath}
where $s$ is a positive integer.
Then we have (using $n = 2p$),
\begin{displaymath}
2p = sp\frac{(q^{d-1} - 1)}{(q^{2} - 1)},
\end{displaymath}
which gives
\begin{displaymath}
2(q^{2} - 1) = s(q^{d-1} - 1),
\end{displaymath}
or, dividing the common factor of $q - 1$ from both sides,
\begin{displaymath}
2(q + 1) = s(1 + q + \cdots + q^{d-2}).
\end{displaymath}
But $s\geq 1$, so we obtain
\begin{eqnarray*}
2(q + 1 ) &   =  & s(1 + q + \cdots + q^{d-2}) \\
          & \geq & 1 + q + \cdots + q^{d-2}.
\end{eqnarray*}
Now subtracting $q - 1$ from both sides of this inequality yields
\begin{displaymath}
q + 1 \geq q^{2} + \cdots + q^{d-2}.
\end{displaymath}
This can occur only if $d = 4$ so that there is only one summand on the right hand side,
in which case we get $q + 1 \geq q^{2}$.
But this is impossible, since $q\geq 2$ then implies that
\begin{displaymath}
0 \geq q^2 - q - 1 > q^2 - 2q + 1 = (q - 1)^{2} \geq 1.
\end{displaymath}
Therefore, this case cannot occur.

\paragraph{Case 4}

Next, we consider the action of $\PSL{7}{q}$ on  $3$- or $4$-dimensional subspaces,
where the degree is
\begin{displaymath}
n = \frac{(q^7 - 1)(q^6 - 1)(q^5 - 1)}{(q^3 - 1)(q^2 - 1)(q - 1)},
\end{displaymath}
and where $p = \frac{(q^7 - 1)}{(q - 1)}$.

Using $2p = n$, we obtain
\begin{displaymath}
2\frac{(q^7 - 1)}{(q - 1)} = \frac{(q^7 - 1)(q^6 - 1)(q^5 - 1)}{(q^3 - 1)(q^2 - 1)(q - 1)},
\end{displaymath}
which yields
\begin{displaymath}
2 = \frac{(q^6 - 1)(q^5 - 1)}{(q^3 - 1)(q^2 - 1)},
\end{displaymath}
or, equivalently,
\begin{displaymath}
2(q^3 - 1)(q^2 - 1) = (q^6 - 1)(q^5 - 1).
\end{displaymath}
But, since $q\geq 2$, we have
\begin{eqnarray*}
2(q^3 - 1)(q^2 - 1) &  <  & (8q^3 - 1)(8q^2 - 1) \\
                    &\leq & (q^6 - 1)(q^5 - 1),
\end{eqnarray*}
a contradiction.

\paragraph{Case 5}
Next, $\PSL{d}{q}$ acts on incident point-hyperplane pairs, with degree equal to
\begin{displaymath}
n = \frac{(q^d - 1)(q^{d-1} - 1)}{(q-1)^2},
\end{displaymath}
where $d\geq 3$
and $p$ divides $\frac{(q^d - 1)}{(q-1)}$
(and $G$ contains a graph automorphism).

Since $p$ divides $\frac{(q^d - 1)}{(q-1)}$, write
\begin{displaymath}
sp = \frac{(q^d - 1)}{(q-1)},
\end{displaymath}
where $s$ is a positive integer.
Then we have
\begin{displaymath}
2p = sp\frac{(q^{d-1} - 1)}{(q - 1)},
\end{displaymath}
so that
\begin{displaymath}
2(q - 1) = s(q^{d-1} - 1) = s(q - 1)(1 + q + \cdots + q^{d-2}).
\end{displaymath}
Dividing both sides by $q - 1$, we obtain
\begin{displaymath}
2 = s(1 + q + \cdots + q^{d-2}).
\end{displaymath}
Since $s\geq 1$, it follows that
\begin{displaymath}
2  \geq 1 + q + \cdots + q^{d-2},
\end{displaymath}
which is impossible, since $d\geq 3$.

\paragraph{Case 6}

Next, there is an action of $\PSL{d}{q}$ on non-incident point-hyperplane pairs, where the degree is
\begin{displaymath}
n = \frac{q^{d-1}(q^d - 1)}{(q-1)},
\end{displaymath}
and where $d\geq 3$
and $p$ divides $\frac{(q^d - 1)}{(q-1)}$
(and $G$ contains a graph automorphism).

Since $p$ must divide $\frac{(q^d - 1)}{(q-1)}$, write
\begin{displaymath}
sp = \frac{(q^d - 1)}{(q-1)},
\end{displaymath}
for some positive integer $s$.
Then we have $2p = n = spq^{d-1}$ or
\begin{displaymath}
2 = sq^{d-1} \geq q^{d-1} \geq q^{2} \geq 4,
\end{displaymath}
a contradiction.
Therefore, this case cannot occur.

\paragraph{Case 7}

Next, the action of $\PSL{4}{3}$ on the cosets of its subgroup of shape $\PSp{4}{3}.2$.
In this case, the degree is $117$ and $p = 13$.

This case does not occur, since $2p = 26\neq 117$.

Finally, there are several actions of $\PSL{2}{q}$, with $q\geq 4$, as follows.

\paragraph{Case 8}

The action on cosets of a dihedral subgroup of degree
\begin{displaymath}
n = \frac{q(q\pm 1)}{2},
\end{displaymath}
where $p = q$ if $q$ is odd, and $p = q\pm 1$ if $q$ is even.
In this case, $G = \PGL{2}{q}$, for $q = 7, 11$.

Suppose first that $q$ is odd, so $p = q$.
Then we get $2p = n = \frac{p(p\pm 1)}{2}$,
so that $4p = p(p\pm 1)$, and hence, $4 = p\pm 1$.
This implies that $p = 3$ or $p = 5$.

Now suppose that $q$ is even.
Then $p = q\pm 1$, so $2p = n = \frac{qp}{2}$
and $q = 4$ giving $p = 3$ or $p = 5$.
In this way, we obtain Case (5) of Proposition~\ref{prop:prim2p}.

\paragraph{Case 9}

The action of $\PSL{2}{q}$ on cosets of $\PGL{2}{\sqrt{q}}$,
with $p$ a divisor of $q+1$,
where $q$ is a square, and the degree is
\begin{displaymath}
n = \frac{\sqrt{q}(q+1)}{f},
\end{displaymath}
where $f = (2, q - 1)$.

Since $p$ divides $q + 1$, there is a positive integer $s$ for which
\begin{displaymath}
sp = q + 1.
\end{displaymath}
Then we have
\begin{displaymath}
2p = n = \frac{sp\sqrt{q}}{f},
\end{displaymath}
or
\begin{displaymath}
2f = s\sqrt{q}.
\end{displaymath}

First suppose that $q$ is even.
Then $q - 1$ is odd, so $f = 1$, and we get
$2 = s\sqrt{q}$ or $4 = s^{2}q$.
Since $q$ is a square, we can write $q = v^{2}$, for some integer $v\geq 2$.
Then we have $4 = s^{2}v^{2} = (sv)^{2}$.
Now, $v\geq 2$ so we have
\begin{displaymath}
4 = (sv)^{2} \geq (2s)^{2} \geq 4s^{2},
\end{displaymath}
so $s = 1$ and $p = q + 1$.
But $2p = \sqrt{q}(q + 1)$, so $2p = p\sqrt{q}$ and hence $\sqrt{q} = 2$.
This implies that $q = 4$.
Thus, $p = 5$ and $n = 10$.
Here, we have Case (6) in Proposition~\ref{prop:prim2p}.

Now suppose that $q$ is odd.
Then $f = 2$ and we have
\begin{displaymath}
2p = \frac{\sqrt{q}(q + 1)}{2},
\end{displaymath}
or
\begin{displaymath}
4p = \sqrt{q}(q + 1).
\end{displaymath}
Again, let $s$ be a positive integer such that $sp = q + 1$.
Then $4p = sp\sqrt{q}$ so that $4 = s\sqrt{q}$.
Now, $q\geq 4$ since $q$ is a square, so
\begin{displaymath}
4 = s\sqrt{q} \geq s\sqrt{4} = 2s,
\end{displaymath}
which implies that $s\leq 2$, so $s\in\{ 1, 2 \}$.

If $s = 1$, then $p = q + 1$ so $2p = \frac{p\sqrt{q}}{2}$ or $4 = \sqrt{q}$,
and so $q = 16$.
But we supposed that $q$ was odd, so this case does not occur.

If $s = 2$, we get $4p = 2p\sqrt{q}$ so that $\sqrt{q} = 2$; that is, $q = 4$.
Again, since $q$ is odd, this case does not occur either.

\paragraph{Case 10}

The action on cosets of $\alt{5}$, where $p = q$, $q\equiv\pm 1\pmod{10}$,
and the degree is
\begin{displaymath}
n = \frac{q(q^2 - 1)}{120},
\end{displaymath}
and $q\leq 109$.

From $2p = n$ and $p = q$, we obtain
\begin{displaymath}
2p = \frac{p(p^{2} - 1)}{120},
\end{displaymath}
from which it follows that
\begin{displaymath}
240p = p(p^{2} - 1),
\end{displaymath}
and hence, $p^{2} = 241$, which has no integer solutions.

\paragraph{Case 11}

The action on cosets of $\symm{4}$, where $p = q$, $q\equiv\pm 1\pmod{8}$,
and the degree is
\begin{displaymath}
n = \frac{q(q^2 - 1)}{48},
\end{displaymath}
where $q\leq 47$.

As in the previous case, we find that $p^{2} = 97$, which has no integer solutions.

\paragraph{Case 12}

The action on cosets of $\alt{4}$, where $p = q$, $q\equiv 3\pmod{8}$,
and the degree is
\begin{displaymath}
n = \frac{q(q^2 - 1)}{24},
\end{displaymath}
where $q\leq 19$.

Here, using $p = q$ and $2p = n$, we obtain $p^{2} = 49$, and so $p = 7$.
However, $7\not\equiv 3\pmod{8}$, so this case does not arise either.

\subsubsection{$\mathrm{PSp}$}

\paragraph{Case 13}
There is an action of the group $\PSp{2d}{q}$ on lines (or, on totally isotropic $2$-dimensional subspaces, in case $d = 2$),
with $p$ a divisor of $q^{d} + 1$ and $d$ a power of $2$, of degree
\begin{displaymath}
n = \frac{(q^{2d} - 1)}{(q - 1)}.
\end{displaymath}

Writing $q^{d} + 1 = sp$, for some positive integer $s$, and using $2p = n$, we obtain
\begin{eqnarray*}
2p & = & \frac{(q^{2d} - 1)}{(q - 1)} \\
   & = & \frac{(q^{d} + 1)(q^{d} - 1)}{(q - 1)} \\
   & = & sp \frac{(q^{d} - 1)}{(q - 1)}.
\end{eqnarray*}
Hence, we have
\begin{displaymath}
2 = s\frac{(q^{d} - 1)}{(q - 1)} = s(1 + q + \cdots q^{d-1}).
\end{displaymath}
As before, this is impossible for $d\geq 2$, since $s\geq 1$.

\subsubsection{$\mathrm{Sp}$}

\paragraph{Case 14}

For even $q$, there is an action of $\Sp{2d}{q}$ of degree
\begin{displaymath}
n = \frac{q^{d}(q^{d}\pm 1)}{2},
\end{displaymath}
where $p = q^{d}\pm 1$.
If $p = q^{d} + 1$, then $d$ is a power of $2$.
If $p = q^{d} - 1$, then $q = 2$ and $d$ is prime.

First suppose that $p = q^{d} + 1$, so that
\begin{displaymath}
2(q^{d} + 1) = \frac{q^{d}(q^{d} + 1)}{2},
\end{displaymath}
which yields
\begin{displaymath}
4 = q^{d}.
\end{displaymath}
Therefore, $q = d = 2$, and so $p = 5$ and $n = 10$.
This yields Case (7) of Proposition~\ref{prop:prim2p}.

If, instead, $p = q^{d} - 1$, then we obtain
\begin{displaymath}
2(q^{d} - 1) = \frac{q^{d}(q^{d} - 1)}{2},
\end{displaymath}
so that, again,
\begin{displaymath}
4 = q^{d},
\end{displaymath}
and so $p = 3$ and $n = 6$,
and we obtain Case (7) of Proposition~\ref{prop:prim2p} again.

\paragraph{Case 15}

There is an action of $\Sp{4}{q}$, for even $q$, of degree
\begin{displaymath}
n = \frac{q^{2}(q^{2} + 1)}{2},
\end{displaymath}
where $p = q^{2} + 1$.

In this case we get, from $n = 2p$,
\begin{displaymath}
2( q^{2} + 1 ) = \frac{q^{2}(q^{2} + 1)}{2},
\end{displaymath}
from which it follows that $q^{2} = 4$, and hence, $q = 2$ and $n = 10$.
Here we have Case (7) of Proposition~\ref{prop:prim2p} again.

\subsubsection{PSU}

\paragraph{Case 16}

The group $\PSU{d}{q}$ acts on singular $1$-subspaces, for prime $d\geq 3$,
with degree
\begin{displaymath}
n = \frac{(q^{d} + 1)(q^{d-1} - 1)}{(q^{2} - 1)},
\end{displaymath}
where $p$ divides $\frac{(q^{d} + 1)}{(q + 1)}$.

Since $p$ divides $\frac{(q^{d} + 1)}{(q + 1)}$, there is a positive integer $s$ for which
\begin{displaymath}
sp = \frac{(q^{d} + 1)}{(q + 1)}.
\end{displaymath}
Then, from $2p = n$, we obtain
\begin{displaymath}
2p = sp\frac{(q^{d-1} - 1)}{(q - 1)},
\end{displaymath}
so that
\begin{displaymath}
2(q - 1) = s(q^{d-1} - 1) = s(q - 1)(1 + q + \cdots + q^{d-2}).
\end{displaymath}
Dividing out the common factor of $q - 1$, we obtain
\begin{displaymath}
2 = s(1 + q + \cdots + q^{d-2}).
\end{displaymath}
Hence, as $s\geq 1$, we obtain
\begin{eqnarray*}
2 &   =  & s(1 + q + \cdots + q^{d-2}) \\
  & \geq & 1 + q + \cdots + q^{d-2} \\
  & \geq & 3,
\end{eqnarray*}
unless $d = 2$.
But we are given that $d\geq 3$, so this case cannot occur.

\subsubsection{$P\Omega$}

\paragraph{Case 17}

There is an action of $\POmega{2d+1}{q}$ on singular $1$-subspaces,
with degree
\begin{displaymath}
n = \frac{(q^{d} - 1)(q^{d-1} + 1)}{(q-1)},
\end{displaymath}
where $p = \frac{(q^d - 1)}{(q - 1)}$,
and $d > 4$ is prime.

Using $n = 2p$, we obtain
\begin{displaymath}
2\frac{(q^{d} - 1)}{(q - 1)} = \frac{(q^{d} - 1)(q^{d-1} + 1)}{(q-1)},
\end{displaymath}
which yields $q^{d-1} - 1 = 2$, or $q^{d-1} = 3$.
But $d > 4$, so this is impossible.

\paragraph{Case 18}

There is an action of $\POmega{2d+1}{q}$ on singular $1$-dimensional
subspaces, with degree
\begin{displaymath}
n = \frac{(q^{2d} - 1)}{(q-1)},
\end{displaymath}
where $p$ divides $q^d + 1$ and $d$ is a power of $2$.

Since $p$ divides $q^{d} + 1$, we can write $sp = q^{d} + 1$, for some positive integer $s$.
Then we have
\begin{displaymath}
2p = n = \frac{(q^{2d} - 1)}{(q - 1)} = \frac{(q^{d} + 1)(q^{d} - 1)}{(q - 1)} = sp\frac{(q^{d} - 1)}{(q - 1)}.
\end{displaymath}
Hence,
\begin{displaymath}
2 = s\frac{(q^{d} - 1)}{(q - 1)} = s(1 + q + \cdots + q^{d-1}).
\end{displaymath}
Since $s\geq 1$, we get
\begin{displaymath}
2 \geq 1 + q + \cdots + q^{d-1}.
\end{displaymath}
But, since $d\geq 2$ and $q\geq 2$, this is impossible,
and we conclude that this case cannot occur.

\subsubsection{$\Omega^{+}$}

\paragraph{Case 19}

There is an action of $\OmegaPlus{2d}{2}$ on non-singular subspaces,
for prime $d > 4$, where $p = 2^{d} - 1$, with degree
\begin{displaymath}
n = 2^{d-1}(2^{d} - 1).
\end{displaymath}
In this case, substituting $n = 2p$, we obtain
\begin{displaymath}
2p = 2^{d-1}(2^{d} - 1),
\end{displaymath}
which, for $p = 2^{d} - 1$ gives
\begin{displaymath}
2(2^{d} - 1) = 2^{d-1}(2^{d} - 1),
\end{displaymath}
which yields $2 = 2^{d-1}$. Hence, $d = 2$.
But we began with $d > 4$,
so this case cannot occur.

\subsubsection{$P\Omega^{+}$}

\paragraph{Case 20}

There is an action of $\POmegaPlus{2d}{q}$ with $p$ a divisor of $q^{d} + 1$
and $d\geq 4$ a power of $2$ with degree either
\begin{displaymath}
n = \frac{(q^{d} + 1)(q^{d - 1} - 1)}{(q-1)}
\end{displaymath}
or
\begin{displaymath}
n = q^{d-1}(q^{d} + 1).
\end{displaymath}

Suppose first that we have degree $n = q^{d-1}(q^{d} + 1)$.
Write $sp = q^{d} + 1$, for some positive integer $s$.
Then we have
\begin{displaymath}
2p = n = q^{d-1}(q^{d} + 1) = spq^{d-1},
\end{displaymath}
or
\begin{displaymath}
2 = sq^{d-1}.
\end{displaymath}
Thus, either $s = 1$ and $q^{d-1} = 2$ and so $q = 2$ and $d = 2$,
or else $s = 2$ and we get $q^{d-1} = 1$, so $d = 1$.
Thus, this case cannot occur.

Now assume that the degree $n$ is
\begin{displaymath}
n = \frac{(q^{d} + 1)(q^{d - 1} - 1)}{(q-1)}.
\end{displaymath}
Using $sp = q^{d} + 1$, we obtain
\begin{displaymath}
2p = sp\frac{q^{d-1} - 1)}{(q - 1)} = sp(1 + q + \cdots + q^{d-2}).
\end{displaymath}
Hence, since $s\geq 1$, we have
\begin{displaymath}
2 \geq 1 + q + \cdots + q^{d-2}.
\end{displaymath}
Now, since $d \geq 4$, therefore, this is impossible.

\subsection{Exceptional Socle}
\subsubsection{$Sz(q)$}

\paragraph{Case 21}

Here, the socle is the group $Sz(q)$,
with degree $q^{2} + 1$ and $p\mid q^{2} + 1$, $p > q$, $q = 2^{2m+1}$.

We have
\begin{displaymath}
2p = q^{2} + 1 = 2^{2(2m + 1)} + 1,
\end{displaymath}
which is impossible, since $2^{2(2m+1)} + 1$ is odd, while $2p$ is even.

\subsubsection{${}^2G_{2}(q)$}

\paragraph{Case 22}

In this case, we consider groups with socle
the Ree group $R(q) ={}^{2}G_{2}(q)$ with degree $q^{2} + 1$,
$p\mid q^{2} + q + 1$, $p > \sqrt{n}$, $q = 3^{2m+1}$.

Since $p$ divides $q^{2} + q + 1$, there is a positive integer $s$ such that $sp = q^{2} + q + 1$.
Then we have
\begin{displaymath}
2p = n = q^{2} + 1 = q + sp,
\end{displaymath}
or
\begin{displaymath}
q = p(2 - s).
\end{displaymath}
Hence,
\begin{displaymath}
p(2 - s) = 3^{2m + 1}.
\end{displaymath}
Thus, $p = 3$, so $n = 6$.
But then $6 = 3^{2(2m+1)} + 1$,
so $3^{2(2m+1)} = 5$, a contradiction.
Therefore, this case cannot occur.

\subsection{Sporadic Socle}

\paragraph{Case 23}
The only sporadic simple groups that occur are the Mathieu groups,
of degrees $276, 23, 253, 506, 22, 77, 66, 11, 55$ and $66$,
and the sporadic groups $J_{1}$ of degree $266$ and the Conway
group $Co_{2}$ of degree $276$.
Of these, only $22 = 2\cdot 11$ is twice a prime number,
which is Case (2) in Proposition~\ref{prop:prim2p}.

This completes the proof of Proposition~\ref{prop:prim2p}.{\qed}

\bibliographystyle{abbrvurl}
\bibliography{magma}

\begin{thebibliography}{10}

\bibitem{AndruskiewitschGrana2003}
N.~Andruskiewitsch and M.~Gra{\~n}a.
\newblock From racks to pointed {Hopf} algebras.
\newblock {\em Adv. Math.}, 178:177--243, 2003.

\bibitem{Blackburn2012}
S.~R. Blackburn.
\newblock Enumerating finite racks, quandles and kei, March 2012.
\newblock \href {http://arxiv.org/abs/1203.6504v1 [math.GT]}
  {\path{arXiv:1203.6504v1 [math.GT]}}.

\bibitem{Burnside1901}
W.~Burnside.
\newblock On some properties of groups of odd order.
\newblock {\em J. London Math. Soc.}, 33:162--185, 1901.

\bibitem{Carter2010}
J.~S. Carter.
\newblock A survey of quandle ideas, February 2010.
\newblock \href {http://arxiv.org/abs/1002.4429v2 [math.GT]}
  {\path{arXiv:1002.4429v2 [math.GT]}}.

\bibitem{Clauwens2011}
F.~J.~B.~J. Clauwens.
\newblock Small connected quandles, July 2011.
\newblock \href {http://arxiv.org/abs/1011.2456v2 [math.GR]}
  {\path{arXiv:1011.2456v2 [math.GR]}}.

\bibitem{EtingofGuralnickSoloviev2001}
P.~Etingof, R.~Guralnick, and A.~Soloviev.
\newblock Indecomposable set-theoretical solutions to the quantum
  {Y}ang-{B}axter equation on a set with a prime number of elements.
\newblock {\em J. Algebra}, 242(2):709--719, 2001.
\newblock \href {http://arxiv.org/abs/0007170v1 [math.QA]}
  {\path{arXiv:0007170v1 [math.QA]}}.

\bibitem{FennRourke1992}
R.~Fenn and C.~Rourke.
\newblock Racks and links in codimension two.
\newblock {\em J. Knot Theory Ramifications}, 1:343--406, 1992.

\bibitem{FermanNowikTeicher2008}
A.~Ferman, T.~Nowik, and M.~Teicher.
\newblock On the structure and automorphism group of finite alexander quandles.
\newblock Preprint, November 2008.
\newblock \href {http://arxiv.org/abs/0811.4211v1 [math.GT]}
  {\path{arXiv:0811.4211v1 [math.GT]}}.

\bibitem{Grana2004}
M.~Gra{\~n}a.
\newblock Indecomposable racks of order {$p^2$}.
\newblock {\em Beitr\"age Alg. Geom.}, 45(2)::665--676, 2004.

\bibitem{HendersonMacedoNelson2006}
R.~Henderson, T.~Macedo, and S.~Nelson.
\newblock Symbolic computation with finite quandles.
\newblock {\em J. Symb. Comp.}, 41:811--817, 2006.
\newblock \href {http://arxiv.org/abs/0508351v2 [math.GT]}
  {\path{arXiv:0508351v2 [math.GT]}}.

\bibitem{HoNelson2005}
B.~Ho and S.~Nelson.
\newblock Matrices and finite quandles.
\newblock {\em Homology, Homotopy and Applications}, 7(1):197--208, 2005.
\newblock \href {http://arxiv.org/abs/0412417v3 [math.GT]}
  {\path{arXiv:0412417v3 [math.GT]}}.

\bibitem{MO75672}
D.~Holt and L.~Vendramin.
\newblock Finite simple groups and conjugacy classes with 2p elements.
\newblock MathOverflow.
\newblock \url{http://mathoverflow.net/questions/75672} (version: 2011-09-17).
\newblock Available from: \url{http://mathoverflow.net/questions/75672}, \href
  {http://arxiv.org/abs/http://mathoverflow.net/questions/75672}
  {\path{arXiv:http://mathoverflow.net/questions/75672}}.

\bibitem{Joyce1982a}
D.~Joyce.
\newblock A classifying invariant of knots, the knot quandle.
\newblock {\em Journal of Pure and Applied Algebra}, 23(1):37--65, 1982.
\newblock \href {http://dx.doi.org/DOI: 10.1016/0022-4049(82)90077-9}
  {\path{doi:DOI: 10.1016/0022-4049(82)90077-9}}.

\bibitem{Kamada2002}
S.~Kamada.
\newblock Knot invariants derived from quandles and racks.
\newblock {\em Geometry and Topology Monographs}, 4:103--117, 2002.

\bibitem{LiebeckSaxl1985a}
M.~W. Liebeck and J.~Saxl.
\newblock Primitive permutation groups containing an element of large prime
  order.
\newblock {\em J. London Math. Soc.}, 31:250--264, 1985.

\bibitem{Matveev1982}
S.~V. Matveev.
\newblock Distributive groupoids in knot theory.
\newblock {\em Mat. Sbornik (N.S.)}, 119(1):78--88, 1982.

\bibitem{OEIS:A181769}
J.~McCarron.
\newblock A181769 (number of isomorphism classes of quandles of order n).
\newblock \url{http://oeis.org/A181769}.

\bibitem{OEIS:A181771}
J.~McCarron.
\newblock A181771 (number of isomorphism classes of connected quandles of order
  n).
\newblock \url{http://oeis.org/A181771}.

\bibitem{McCarron2012a}
J.~McCarron.
\newblock Small homogeneous quandles.
\newblock In {\em Proceedings ISSAC 2012}, Grenoble, France, July 2012.
\newblock (to appear).

\bibitem{OEIS}
N.~J. Sloane and \textit{et al}.
\newblock Online encyclopedia of integer sequences.
\newblock \url{http://oeis.org}.

\bibitem{Vendramin2011}
L.~Vendramin.
\newblock On the classification of quandles of low order, June 2011.
\newblock \href {http://arxiv.org/abs/1105.5341v2 [math.GT]}
  {\path{arXiv:1105.5341v2 [math.GT]}}.

\bibitem{Vendramin2012}
L.~Vendramin.
\newblock On the classification of quandles of low order.
\newblock {\em J. Knot Theory Ramifications}, 21, August 2012.
\newblock \href {http://dx.doi.org/10.1142/S0218216512500885}
  {\path{doi:10.1142/S0218216512500885}}.

\end{thebibliography}

\end{document}